\documentclass{amsart}
\usepackage{pgf,pgfarrows,pgfnodes,pgfautomata,pgfheaps,pgfshade,hyperref, amssymb,enumerate,amsmath}
\usepackage[all]{xy}
\usepackage[capitalize]{cleveref}
\usepackage{mathtools}
\usepackage[shortlabels]{enumitem}
\usepackage{stmaryrd}
\usepackage{tcolorbox}

\definecolor{verde}{cmyk}{0.7,0,0.8,0.1}

\newtheorem{theorem}{Theorem}[section]

\newtheorem*{theorem*}{theorem}
\newtheorem{proposition}[theorem]{Proposition}
\newtheorem{lemma}[theorem]{Lemma}
\newtheorem{Lema}[theorem]{Lemma}
\newtheorem{corollary}[theorem]{Corollary}

\newtheorem{remark}[theorem]{Remark}

\usepackage{euscript,epsfig}
\usepackage{tikz}
\usetikzlibrary{graphs}
\usetikzlibrary[graphs]
\usetikzlibrary{arrows}
\usepackage{tikz-cd}
\usetikzlibrary{decorations.markings}
\usepackage[colorinlistoftodos]{todonotes}

\newcommand{\N}{\mathbb{N}}

\newcommand{\sub}{\subseteq}

\newcommand{\aut}{\operatorname{Aut}}

\DeclareMathOperator{\sym}{Sym}

\def\Ker{{\rm Ker}}
\def\aut{{\rm Aut}}

\author[]{Mar\'{\i}a Isabel Cortez}
\address{Facultad de Matem\'aticas, Pontificia Universidad Cat\'olica de Chile. Edificio Rolando Chuaqui, Campus San Joaquín. Avda. Vicuña Mackenna 4860, Macul, Chile.}
\email{maria.cortez@uc.cl}

\author[]{Vicente Urria}
\address{Facultad de Matem\'aticas, Pontificia Universidad Cat\'olica de Chile. Edificio Rolando Chuaqui, Campus San Joaquín. Avda. Vicuña Mackenna 4860, Macul, Chile.}
\email{vicente.urria@uc.cl}



\begin{document}

\title{Stabilized automorphism groups and full groups of odometers.}

 \maketitle

\begin{abstract}
  In this article, we show that the stabilized automorphism group of   free exact  odometers arising from actions of  finitely generated   residually finite groups coincides with the topological full group of the odometer acting on itself by right multiplication. We then prove that two   free exact   odometers have isomorphic stabilized automorphism groups if and only if they have isomorphic clopen subgroups of the same index. As a consequence, continuous orbit equivalence implies isomorphic stabilized automorphism groups, while for   free  $\mathbb{Z}^d$-odometers, isomorphic stabilized automorphism groups imply orbit equivalence. In general, neither continuous orbit equivalence nor orbit equivalence is equivalent to having isomorphic stabilized automorphism groups.
  \end{abstract}
 
 \section{Introduction}
Stabilized automorphism groups were introduced by Hartman, Kra, and Schmieding in \cite{HKS22} in the setting of subshifts of finite type. In \cite{Sch22}, Schmieding provided a complete characterization of the isomorphism class of the stabilized automorphism groups associated with full shifts. Subsequently, Jones-Baro \cite{JB24} initiated the study of these algebraic invariants in the context of minimal Cantor systems, giving explicit des\-crip\-tions of the stabilized automorphism groups of odometers and Toeplitz subshifts, and showing that these groups serve as complete invariants for conjugacy in certain classes of odometers. These results were later extended by Espinoza and Jones-Baro in \cite{BJ24} to include minimal and certain transitive systems, where they demonstrated that having the same rational eigenvalues is a necessary condition for the stabilized automorphism groups to be isomorphic. More recently, in \cite{Sal23}, Salo introduced the notion of stabilized automorphism groups for actions of residually finite groups and described these groups for shifts of finite type.

In this article, we study the stabilized automorphism group of odometers endowed with a group structure, where the acting group is residually finite. We show that, in this setting, the stabilized automorphism group coincides with the topological full group of the odometer acting on itself by right multiplication. Furthermore, we prove that two odometers have isomorphic stabilized automorphism groups if and only if they possess isomorphic clopen subgroups of the same index. As a consequence, continuously orbit equivalent odometers necessarily have isomorphic stabilized automorphism groups; however, the converse does not hold. Indeed, the results in \cite{KS23} imply there exist odometers with isomorphic stabilized automorphism groups, one arising from the action of a finitely generated amenable group and the other from the action of a finitely generated non-amenable group.

In the case of $\mathbb{Z}^d$-odometers, an isomorphism of stabilized automorphism groups implies orbit equivalence. However, beyond $\mathbb{Z}$-odometers, neither continuous orbit equivalence nor orbit equivalence is equivalent to having isomorphic stabilized automorphism groups: we provide examples of non-continuously orbit equivalent $\mathbb{Z}^2$-odometers with isomorphic stabilized automorphism groups, as well as orbit equivalent $\mathbb{Z}^2$-odometers with non-isomorphic stabilized automorphism groups.

\medskip

    Before stating our main results, we briefly recall some basic notions.  
Let $G$ be a residually finite group. We say that $X$ is an \emph{odometer of $G$} if it is the inverse limit of the quotient sets associated with a decreasing sequence of finite index subgroups of $G$. If these subgroups are normal, then $X$ is called an \emph{exact odometer} of $G$. We denote by $\alpha_X$ the action of $G$ on $X$ given by left multiplication, and we refer to the dynamical system $(X,\alpha_X,G)$ as an odometer or an exact odometer accordingly.  

Note that if $X$ is an exact odometer such that the intersection of the defining normal finite index subgroups is trivial, then the action $\alpha_X$ is free. In this case, we say that both $X$ and $(X,\alpha_X,G)$ are \emph{free exact odometers}. Since an exact odometer $X$ is a group, it also acts by left and right multiplication on itself, and we denote these actions by $L_X$ and $R_X$, respectively.

More generally, if $(X,\alpha,G)$ is any topological dynamical system, we denote by $[[\alpha]]$ and $\aut^{\infty}(X,\alpha,G)$ the topological full group and the stabilized automorphism group of $(X,\alpha,G)$, respectively. Further details on these notions will be given in the next sections.

We say that the topological dynamical systems $(X,\alpha,G)$ and $(Y,\beta,H)$ are 
\emph{orbit equivalent} if there exists a homeomorphism $\varphi \colon X \to Y$ that 
induces a bijection between the orbits of $\alpha$ and $\beta$. They are said to be 
\emph{continuously orbit equivalent} if they are orbit equivalent via a homeomorphism 
$\varphi \colon X \to Y$  with the property that, for every $g \in G$ and every $x \in X$, there exist a clopen neighborhood $U$ of $x$ and an element $h \in H$ such that $\varphi(\alpha^g(y)) = \beta^h(\varphi(y))$, for all $y\in U$. See \cite{CM16, GMPS10, GPS19, M11} for further details on orbit equivalence.

\medskip

Our main results are the following:

\begin{theorem}\label{main_full_group}
   Let $G$ and $H$ be infinite finitely generated residually finite groups.    Let $X$ and $Y$ be   free exact odometers of $G$ and $H$, respectively.       Then the following statements are equivalent:
 \begin{enumerate}
 \item $[[L_X]]\cong [[L_Y]]$.
 \item $[[R_X]]\cong [[R_Y]]$.
 \item There exist clopen subgroups $U_1\subseteq X$ and $U_2\subseteq Y$ such that $U_1$ and $U_2$ are isomorphic as topological groups and $[X:U_1]=[Y:U_2]<\infty$.
 \end{enumerate}
 \end{theorem}

\begin{theorem}\label{main-characterization} 
   Let $G$ be a countable infinite residually finite group, and   let $X$ be a free exact odometer of $G$. 
    Then $\aut^{\infty}(X,\alpha_{X},G)$ coincides with $[[R_X]]$.
 \end{theorem}

As a consequence of Theorems  \ref{main_full_group} and \ref{main-characterization},  
we obtain the following characterization of isomorphisms between stabilized 
automorphism groups.

\begin{theorem}\label{main-1}
  Let $G$ and $H$ be infinite, finitely generated, residually finite groups,    and let $X$ and $Y$ be free exact odometers of $G$ and $H$, respectively.  Then the following statements are equivalent:
  \begin{enumerate}
    \item $\aut^{\infty}(X, \alpha_{X}, G) \cong \aut^{\infty}(Y, \alpha_{Y}, H)$.
    \item $[[L_X]]\cong [[L_Y]]$.
    \item $[[R_X]]\cong [[R_Y]]$,
    \item There exist clopen subgroups $U_1\subseteq X$ and $U_2\subseteq Y$ such that $U_1$ and $U_2$ are isomorphic as topological groups and $[X:U_1]=[Y:U_2]<\infty$.
  \end{enumerate} 
\end{theorem}
 
As a consequence of Theorem~\ref{main-1} together with \cite[Remark~3]{M11}, 
an isomorphism between $\aut^{\infty}(X,\alpha_X,G)$ and $\aut^{\infty}(Y,\alpha_Y,H)$ 
is equivalent to continuous orbit equivalence between the systems 
$(X,L_X,X)$ and $(Y,L_Y,Y)$, and also between $(X,R_X,X)$ and $(Y,R_Y,Y)$. 

The following result is a consequence of the preceding theorems together with 
the characterization of orbit equivalence for $\mathbb{Z}^d$-actions obtained by 
Giordano, Matui, Putnam, and Skau in  \cite{GMPS10}.

\begin{corollary}\label{main_OE}
Let $d_1,d_2 \in \mathbb{N}$,   and let $X_1$ and $X_2$ be free odometers  of $\mathbb{Z}^{d_1}$ and $\mathbb{Z}^{d_2}$, respectively. Then:
\begin{enumerate}
  \item If  $\aut^{\infty}(X_1, \alpha_{X_1},\mathbb{Z}^{d_1})\cong \aut^{\infty}(X_2, \alpha_{X_2},\mathbb{Z}^{d_2})$, then the dynamical systems  
  $(X_1, \alpha_{X_1},\mathbb{Z}^{d_1})$ and $(X_2, \alpha_{X_2},\mathbb{Z}^{d_2})$ are orbit equivalent.
  
  \item If  
  $(X_1, \alpha_{X_1},\mathbb{Z}^{d_1})$ and $(X_2, \alpha_{X_2},\mathbb{Z}^{d_2})$   are continuously 
  orbit equivalent, then   
  $\aut^{\infty}(X_1, \alpha_{X_1},\mathbb{Z}^{d_1}) \cong\aut^{\infty}(X_2, \alpha_{X_2},\mathbb{Z}^{d_2})$. 
\end{enumerate} 
\end{corollary}


  The paper is organized as follows. In Section~\ref{background} we introduce the notions of odometer, exact odometer, topological full group, automorphism group, and stabilized automorphism group of a dynamical system, with particular emphasis on the case of odometers. In Section~\ref{sec:full_group_odometer}, we introduce the actions $L_X$ and $R_X$ associated with an exact odometer $X$ and show that the closures in $\mathrm{Homeo}(X)$ of $[[\alpha_X]]_n$ and $[[\beta_X]]_n$ coincide with $[[L_X]]_n$ and $[[R_X]]_n$, respectively, where the subscript $n$ denotes the elements of the corresponding full groups that agree with the action of an element of the respective acting group on each cylinder set of level $n$. Using this fact together with the spatial nature of isomorphisms between topological full groups, we prove Theorem~\ref{main_full_group}. In Section~\ref{sec:stabilized} we show that the stabilized automorphism group $\aut^{\infty}(X,\alpha_X,G)$ of a free exact odometer $X$ of $G$ is equal to the union $\bigcup_{n\in\mathbb{N}}\aut(X,\alpha_X|_{\Gamma_n},\Gamma_n)$, where $(\Gamma_n)_{n\in\mathbb{N}}$ is a scale defining $X$ and $\alpha_X|_{\Gamma_n}$ denotes the restriction of $\alpha_X$ to $\Gamma_n$. We also prove that $\aut(X,\alpha_X|_{\Gamma_n},\Gamma_n)$ coincides with the closure of $[[\alpha_X]]_n$, which, together with the results of the previous section, allows us to establish Theorem~\ref{main-characterization} (see Proposition~\ref{like-full-group}). In the same section, we show that the amenability of the stabilized automorphism group of  $(X,\alpha_X,G)$ is equivalent to the amenability of $X$ (which is not equivalent to the amenability of the acting group $G$; see Remark~\ref{non-amenability}). Finally, in Section~\ref{sec:invariance} we study the connections between the different notions of orbit equivalence and the stabilized automorphism group, and we provide examples showing that Corollary~\ref{main_OE} is optimal, which is summarized in Proposition~\ref{Summary}.

\section{Definitions and background}\label{background}

In this article, a (topological) dynamical system refers to a   triple $(X,\alpha,G)$, where $X$ is a compact metric space and  $\alpha: G\times X\to X$ is a left action of the group $G$ on $X$ by homeomorphisms. We write $g \cdot x$ or $\alpha^g(x)$ for the image of $x \in X$ under the action of $g \in G$. The system $(X,\alpha,G)$, or equivalently the action $\alpha$, is said to be free if $\alpha^g(x)=x$ implies $g=1_G$ for every $x \in X$. Unless it is explicitly stated that the acting group is countable (as in the definition of an odometer), we do not assume it to be so. 

When $X$ is a Cantor set, we refer to   $(X,\alpha,G)$ as a \textit{Cantor dynamical system} (or simply a Cantor system).


The symmetric group on a set $F$ is denoted by $\sym(F)$. When $F$ has cardinality $n \in \mathbb{N}$, we may write either $\sym(F)$ or $\sym(n)$.

\subsection{Odometers} Most of the material in this section appears in \cite{CP08} and, for $\mathbb{Z}$-actions, in \cite{Dow05}.

Let $G$ be a countable  group and let $(\Gamma_n)_{n\in\mathbb{N}}$ be a decreasing sequence of finite index subgroups of $G$. For every $n\in\mathbb{N}$, consider $G/\Gamma_n$ the set of left cosets of  $\Gamma_n$ in $G$, and $\tau_n:G/\Gamma_{n+1}\to G/\Gamma_n$ the map given by $\tau_n(g\Gamma_{n+1})=g\Gamma_n$, for every $g\in G$. The {\it $G$-odometer} or odometer associated to $(\Gamma_n)_{n\in\mathbb{N}}$ is the space 
$$
G_{(\Gamma_n)}=\{(x_n)_{n\in\mathbb{N}}\in \prod_{n\in\mathbb{N}}G/\Gamma_n: \tau_n(x_{n+1})=x_n, \mbox{ for every } n\in\mathbb{N}\}.
$$
This space is totally disconnected if we endow every $G/\Gamma_n$ with the discrete topology, $\prod_{n\in\mathbb{N}}G/\Gamma_n$ with the product topology, and $G_{(\Gamma_n)}$ with the induced topology. The  space $G_{(\Gamma_n)}$ is a Cantor set if and only if $G_{(\Gamma_n)}$ is infinite, which is the case if and only if the sequence  $(\Gamma_n)_{n\in\mathbb{N}}$ does not stabilize. The topology of $G_{(\Gamma_n)}$ is induced by the metric
$$
d((x_n)_{n\in\mathbb{N}}, (y_n)_{n\in\mathbb{N}})=\frac{1}{2^{\min\{n\in \mathbb{N}: x_n\neq y_n\}}}, \forall (x_n)_{n\in\mathbb{N}}, (y_n)_{n\in\mathbb{N}}\in G_{(\Gamma_n)}.
$$
Observe that the ball of center $x=(x_n)_{n\in\mathbb{N}}\in G_{(\Gamma_n)}$ and radius $\frac{1}{2^k}$ is the set
$$
[x]_k=\{(y_n)_{n\in\mathbb{N}}: y_k=x_k\}.
$$
To refer to the set $[x]_n$, we will use the notation $[x]_n$,  $[x_n]_n$ and $[g]_n$ interchangeably throughout, where $g\in G$ is such that $g\Gamma_n=x_n$. In other words,  if $a\in G/\Gamma_n$ and $g\in G$ is such that $g\Gamma_n=a$, then the sets $[g]_n$ and $[a]_n$ are different notations to denote  the ball of radius $\frac{1}{2^n}$ centered at the elements $x=(x_k)_{k\in\mathbb{N}}\in G_{(\Gamma_n)}$ such that $x_n=a$.   We call the balls of radius $\frac{1}{2^{n}}$ the {\it cylinder sets of level $n$}.

The group $G$ acts on every set $G/\Gamma_n$ by left multiplication: 
$$
g\cdot (h\Gamma_n)=gh\Gamma_n, \mbox{ for every } g,h\in G.
$$
This action extends to  an action $\alpha_{(\Gamma_n)}: G\times G_{(\Gamma_n)}\to G_{(\Gamma_n)}$ coordinate wise:  
$$
\alpha^g_{(\Gamma_n)}(x)=g\cdot x=(g\cdot x_n)_{n\in\mathbb{N}}, \mbox{ for every } x=(x_n)_{n\in\mathbb{N}}\in G_{(\Gamma_n)} \mbox{ and } g\in G. 
$$
 We refer to the system  $(G_{(\Gamma_n)}, \alpha_{(\Gamma_n)}, G)$  as the   {\it left odometer} or odometer a\-sso\-cia\-ted with $(\Gamma_n)_{n\in\mathbb{N}}$.  It is known that $(G_{(\Gamma_n)}, \alpha_{(\Gamma_n)}, G)$ is minimal, equicontinuous and uniquely ergodic.   Indeed, the unique invariant probability measure $\mu$ is completely determined by 
$$
\mu([a]_n)=\frac{1}{[G:\Gamma_n]}, \mbox{ for every } a\in G/\Gamma_n \mbox{ and } n\in \mathbb{N}.
$$

\subsection{Exact odometers.}

If the groups $(\Gamma_n)_{n\in \mathbb{N}}$ are normal, then the odometer $G_{(\Gamma_n)}$ is a subgroup of  the product group $\prod_{n\in\mathbb{N}}G/\Gamma_n$.  In this setting, we call $G_{(\Gamma_n)}$, as well as the system $(G_{(\Gamma_n)}, \alpha_{(\Gamma_n)}, G)$, the \textit{exact odometer} associated with $(\Gamma_n)_{n\in\mathbb{N}}$. In this case, the unique invariant probability measure of $(G_{(\Gamma_n)}, \alpha_{(\Gamma_n)}, G)$ corresponds to the Haar measure of $G_{(\Gamma_n)}$. Furthermore,  the map $\tau:G\to G_{(\Gamma_n)}$ defined by $\tau(g)=(g\Gamma_n)_{n\in\mathbb{N}}$, for every $g\in G$, is a group homomorphism such that $\tau(G)$ is dense in $G_{(\Gamma_n)}$. If in addition $\bigcap_{n\in\mathbb{N}}\Gamma_n=\{1_G\}$, then  the action $\alpha_{(\Gamma_n)}$ is free, $\tau$ is injective,  and we can identify $G$ with $\tau(G)$.  Conversely, if $K$ is a metrizable compact totally disconnected group for which there exists an injective homomorphism $\tau : G \to K$ with dense image, then the dynamical system given by the action of $G$ on $K$ by left multiplication is conjugate to a free exact odometer (see \cite[Lemma~2.1]{CCG24}).

  If $G_{(\Gamma_n)}$ is an exact odometer, then $G$ also acts on 
$G_{(\Gamma_n)}$ by right multiplication:   $\beta_{(\Gamma_n)}: G\times G_{(\Gamma_n)}\to G_{(\Gamma_n)}$ is given by
\[
\beta^g_{(\Gamma_n)}(x)=g \cdot x \;=\; x \tau(g)^{-1}, 
\quad \text{for every } x \in G_{(\Gamma_n)} \text{ and } g \in G.
\]
 We refer to the system $(G_{(\Gamma_n)}, \beta_{(\Gamma_n)}, G)$  as the 
\emph{right odometer}  associated with 
$(\Gamma_n)_{n \in \mathbb{N}}$. The map 
\(\psi : G_{(\Gamma_n)} \to G_{(\Gamma_n)}\) defined by 
\(\psi(x) = x^{-1}\) for every $x \in G_{(\Gamma_n)}$ is a conjugacy between 
the left and right odometers.  

\medskip

From now on, we say that $(\Gamma_n)_{n \in \mathbb{N}}$ is a \emph{scale} 
if it is a decreasing sequence of normal finite index subgroups of $G$ with 
trivial intersection. Thus, if $(\Gamma_n)_{n\in\mathbb{N}}$ is a scale  then the dynamical system
$(G_{(\Gamma_n)}, \alpha_{(\Gamma_n)}, G)$ is a free eaxct odometer.  We will use the term \textit{free exact odometer} interchangeably to refer to $G_{(\Gamma_n)}$, $(G_{(\Gamma_n)}, \alpha_{(\Gamma_n)}, G)$, and $(G_{(\Gamma_n)}, \beta_{(\Gamma_n)}, G)$.

If, for simplicity, we set $X = G_{(\Gamma_n)}$, then we write $\alpha_X$ and $\beta_X$ in place of $\alpha_{(\Gamma_n)}$ and $\beta_{(\Gamma_n)}$, respectively.


\subsection{Full groups} Let   $(X, \alpha, G)$ be a Cantor dynamical system. 
The \emph{full group}   $[\alpha]$  of $(X, \alpha, G)$ is defined by
\[
[\alpha] = \{ f \in \mathrm{Homeo}(X) : 
   \text{for every } x \in X \text{ there exists } g \in G 
   \text{ such that } f(x) = g \cdot x \}.
\]
The \emph{topological full group}   $[[\alpha]]$ of $(X, \alpha, G)$ consists of all 
$f \in [\alpha]$ for which there exist a finite clopen partition 
$\{A_1, \dots, A_n\}$ of $X$ and elements $g_1, \dots, g_n \in G$ such that 
\[
f|_{A_i}(y) = g_i \cdot y, 
\quad \text{for every } y \in A_i \text{ and } 1 \leq i \leq n.
\]

\subsubsection{Full groups for odometers}

Let $(\Gamma_n)_{n \in \mathbb{N}}$ be a scale, and let $X$ be the a\-sso\-cia\-ted free exact odometer. 
 For $S\in \{\alpha_X, \beta_X\}$ and  $n\in \mathbb{N}$, define 
\[
[[S]]_{n} = \bigl\{ f \in [[T]] : 
   \forall a \in G/\Gamma_n, \; \exists g_a \in G 
   \text{ such that } f|_{[a]_n}(x) = S^{g_a}(x), 
   \; \forall x \in [a]_n \bigr\}.
\]

Note that 
\[
[[S]] = \bigcup_{n \in \mathbb{N}} [[S]]_{n}.
\]

If $G$ is abelian, then $[[\alpha_{X}]]_{n} = [[\beta_{X}]]_{n}$ for every $n \in \mathbb{N}$; 
in particular, the two topological full groups coincide. 
In general, the map $\psi(x) = x^{-1}$ induces an isomorphism 
\[
\alpha : [[\alpha_{X}]] \longrightarrow [[\beta_{X}]], 
\qquad \alpha(f) = \psi \circ f \circ \psi^{-1}, 
\quad \forall f \in [[\alpha_{X}]].
\]
Since $\alpha([[\alpha_{X}]]_{n}) = [[\beta_{X}]]_{n}$, \cite[Proposition~4.6]{CM16} implies that 
\[
[[\beta_{X}]]_{n} \cong [[\alpha_{X}]]_{n} \cong 
\Gamma_n^{G/\Gamma_n} \rtimes \sym(G/\Gamma_n),
\]
for every $n \in \mathbb{N}$,   where the semi-direct product is defined by the homomorphism $\varphi:\sym(G/\Gamma_n)\to \aut(\Gamma_n^{G/\Gamma_n})$ given by $\varphi(\sigma)((\gamma_a)_{a\in G/\Gamma_n})=(\gamma_{\sigma(a)})_{a\in G/\Gamma_n}$, for every $\sigma\in \sym(G/\Gamma_n)$ and $(\gamma_a)_{a\in G/\Gamma_n}\in \Gamma_n^{G/\Gamma_n}$.

\subsection{Automorphism group of an odometer.} 

Let    $(X, \alpha, G)$ be a topological dynamical system. The \emph{group of automorphisms} 
of $(X, \alpha, G)$ is the subgroup of $\mathrm{Homeo}(X)$ defined by
\[
\aut(X, \alpha, G) = \{ f \in \mathrm{Homeo}(X) : 
   f(g \cdot x) = g \cdot f(x), \; \forall x \in X, \; \forall g \in G \}.
\]

Let $(\Gamma_n)_{n \in \mathbb{N}}$ be a scale,   and let $X=G_{(\Gamma_n)}$ be the associated free exact odometer.
For each $a \in X$, define the maps  
$L^a_{X}, R^a_{X} : X \to X$ by
\[
R_{X}^a(x) = x a^{-1}, \quad  L^a_{X}(x) = ax , \quad  \forall x \in X.
\]


The following result was proved in \cite[Lemma~5.9]{DDMP16} for the case 
$G = \mathbb{Z}$ (see also \cite{Aus63}).  

\begin{lemma}\label{characterization_automorphism}
Let $(\Gamma_n)_{n \in \mathbb{N}}$ be a scale,   and let $X=G_{(\Gamma_n)}$ be the associated free exact odometer. Then
\[
f \in \aut(X, \alpha_{X}, G) 
\;\;\iff\;\; f = R^a_{X} \quad \text{for some } a \in X.
\]
Consequently, $\aut(X, \alpha_{X}, G)$ is isomorphic, as an abstract group, 
to $X$, and it coincides with the Ellis semigroup of the right 
odometer $(X, \beta_{X}, G)$.
\end{lemma}

\begin{proof}

It is immediate that $R_X^a \in \aut(X, \alpha_X, G)$ for every $a \in X$.  

Let $\phi \in \aut(X, \alpha_X, G)$, $\varepsilon > 0$, and $x \in X$. 
Set $a = (\phi(1_X))^{-1}$. 
Since $\phi$ is uniformly continuous, there exists $\delta > 0$ such that 
$d(x,y) < \delta$ implies $d(\phi(x), \phi(y)) < \varepsilon/2$. 
Moreover, by minimality, there exists $g \in G$ such that 
\[
d(g \cdot 1_X, x) < \min\{\delta, \tfrac{\varepsilon}{2}\}.
\]  
Hence,
\begin{align*}
d(\phi(x), xa^{-1}) 
  &\leq d(\phi(x), g \cdot a^{-1}) 
       + d(g \cdot a^{-1}, x a^{-1}) \\
  &= d(\phi(x), g \cdot \phi(1_X)) + d(g \cdot 1_X, x) \\
  &\leq d(\phi(x), \phi(g \cdot 1_X)) + \varepsilon/2 \\
  &\leq \varepsilon/2 + \varepsilon/2 = \varepsilon.
\end{align*}
Since $\varepsilon > 0$ is arbitrary, it follows that 
$\phi(x) = x a^{-1} = R_X^a(x)$ for every $x \in X$. 
Thus $\phi = R_X^a$. 
The isomorphism between $X$ and $\aut(X, \alpha_X, G)$ is given by the map 
$t \mapsto R_X^t$.  

Finally, note that for every $t \in X$ the map $R_X^t$ is the uniform limit of 
a sequence $(R_X^{g_n})_{n \in \mathbb{N}}$, where $g_n \in G$ satisfies 
$\tau(g_n) \in [t]_n$ for each $n \in \mathbb{N}$. 
Indeed, for $y \in X$ we have
\[
d(R_X^{g_n}(y), R_X^t(y)) 
  = d(y g_n^{-1}, y t^{-1}) 
  \leq \tfrac{1}{2^n},
\]
since $R_X^{g_n}(y)$ and $R_X^t(y)$ lie in the same cylinder set of level $n$.  

Conversely, as $(X, \beta_X, G)$ is equicontinuous, if $f$ belongs to the Ellis 
semigroup of $(X, \beta_X, G)$, then $f$ is the uniform limit of some sequence 
$(R_X^{g_n})_{n \in \mathbb{N}}$. 
Let $t = \lim_{n \to \infty} R_X^{g_n}(1_X)$. 
Then for $y \in X$,
\[
f(y) = \lim_{n \to \infty} R_X^{g_n}(y) 
      = \lim_{n \to \infty} y g_n^{-1} 
      = \lim_{n \to \infty} y R_X^{g_n}(1_X) 
      = yt 
      = R_X^{t^{-1}}(y).
\]
Therefore, $\aut(X, \alpha_X, G)$ coincides with the Ellis semigroup of $(X, \beta_X, G)$.
\end{proof}

\begin{remark}
  {\rm If $G$ is abelian, then the Ellis semigroup of the right and left  odometers coincide.   }  
\end{remark}
  
\subsection{Stabilized automorphism groups.}   The {\it stabilized automorphism group } of the dynamical system $(X, \alpha, G)$ is defined as
$$\aut^\infty(X, \alpha, G)=\bigcup_{H\in I(G)} \aut(X, \alpha|_H, H),$$
where $I(G)$ is the collection of finite index subgroups of $G$, and $\alpha|_H$ denotes the restriction of  $\alpha$ to $H\times X$, for $H\in I(G)$.  Since $I(G)$ is closed under finite intersection, the set   $\aut^{\infty}(X, \alpha, G)$ is a subgroup of $\mathrm{Homeo}(X)$. 

Observe that $\aut^\infty(X, \alpha, G)$ contains but is
not contained in $\aut(X, \alpha, G)$, because a homeomorphism that commutes with a subaction need not
commute with the whole action.

In the next section, we establish several results that will be essential for the des\-crip\-tion of the stabilized automorphism group of an odometer, which is presented in Section \ref{sec:stabilized}.

\section{Topological full group of the action of the odometer.}\label{sec:full_group_odometer}
 
 Let $G$ be an infinite   countable  residually finite group, and let 
$(\Gamma_n)_{n \in \mathbb{N}}$ be a scale of $G$.   The free exact odometer $X=G_{(\Gamma_n)}$ acts on itself by right and left multiplication: we denote $L_{X}: X\times X\to X$   the action given by 
\[
 L_{X}^{\xi}(x)=\xi \cdot x = \xi x, \quad \text{for all } x, \xi \in X,
\]
and by $R_{X}: X\times X\to  X$ the action defined by
\[
R_{X}^\xi(x)=\xi \cdot x = x \, \xi^{-1}, \quad \text{for all } x, \xi \in X.
\]

The map $\alpha : [[L_X]] \longrightarrow [[R_{X}]]$, given by
\[
 \alpha(f)(x) = (f(x^{-1}))^{-1}, \quad \forall x \in X, f \in [[L_{X}]],
\]
is an isomorphism.


For $S \in \{R_{X},L_{X}\}$ and $n \in \mathbb{N}$, define
\[
[[S]]_{n} = 
\Bigl\{ f \in [[S]] : 
   \forall a \in G/\Gamma_n, \; \exists \xi_a \in G_{X} \text{ such that } 
   f|_{[a]_n} = S^{\xi_a}|_{[a]_n} \Bigr\}.
\]

Observe that
\[
[[S]] = \bigcup_{n \in \mathbb{N}} [[S]]_{n}.
\]


\begin{lemma}\label{closure-1}
Let $G$ be a   countable infinite residually finite group, and let   $X$ be the free exact odometer associated with the scale 
$(\Gamma_n)_{n \in \mathbb{N}}$  of $G$. 
For   every $m \in \mathbb{N}$, we have
$$
\overline{[[\alpha_{X}]]}_m=[[L_{X}]]_m,
$$
and
\[
\overline{[[\beta_{X}]]}_{m} = [[R_{X}]]_{m},
\]
where $\overline{[[\alpha_{X}]]}_{m}$ and $\overline{[[\beta_{X}]]}_{m}$ denote the closure of $[[\alpha_{X}]]_{m}$ and $[[\beta_{X}]]_m$, respectively,  
with respect to the topology of uniform convergence in 
$\mathrm{Homeo}(X)$.
\end{lemma}
 
\begin{proof}
Let $f \in [[R_{X}]]_{m}$. 
Then, for each $a \in G/\Gamma_m$, there exists $\xi_a \in X$ such that
\[
f(x) = x \, \xi_a^{-1}, \quad \forall x \in [a]_m.
\]

For every $k \ge m$ and $a \in G/\Gamma_m$, choose $g_{a,k} \in G$ such that 
$g_{a,k} \in [\xi_a]_k$, and define
\[
f_k(x) = x g_{a,k}^{-1}, \quad \forall x \in [a]_m, \; a \in G/\Gamma_m.
\]
Since $g_{a,k} \in [\xi_a]_m$, we have $f_k([a]_m) = f([a]_m)$.  Since $\{f([a]_m): a\in G/\Gamma_m\}$ is a partition, $f_k$ is a homeomorphism and $f_k \in [[R_{X}]]_{m}$.  

Moreover, for $x \in [a]_m$, both $x g_{a,k}^{-1}$ and $x \xi_a^{-1}$ 
belong to the same cylinder set of level $k$. Therefore, 
\[
d_\infty(f_k, f) \le \frac{1}{2^k},
\]
where $d_\infty$ denotes the uniform distance. 
This shows that 
\[
[[R_{X}]]_{m} \subseteq \overline{[[\beta_{X}]]}_{m}.
\]

Conversely, let $(f_k)_{k \in \mathbb{N}}$ be a sequence in $[[\beta_{X}]]_{m}$ 
converging uniformly to some $f \in \mathrm{Homeo}(X)$. 
For each $k$ and $a\in G/\Gamma_m$, let $g_{a,k}\in G$ be such that
\[
f_k(x) = x g_{a,k}^{-1}, \quad \forall x \in [a]_m.
\]
Let $x \in [a]_m$. Then
\[
f(x) = \lim_{k \to \infty} f_k(x) = \lim_{k \to \infty} x g_{a,k}^{-1} 
= x \, \lim_{k \to \infty} g_{a,k}^{-1}.
\]
Hence, the limit $\xi_a^{-1} := \lim_{k \to \infty} g_{a,k}^{-1}$ exists, 
and $f(x) = x \, \xi_a^{-1}$ for $x \in [a]_m$.  
This proves the reverse inclusion.  

The argument for the left action is analogous.
\end{proof}

\begin{remark}\label{closure-1-remark}
{\rm Lemma \ref{closure-1} implies that the topological full groups   $[[\alpha_{X}]]$ and $[[\beta_X]]$ are dense in $[[L_{X}]]$ and $[[R_{X}]]$, respectively,   with respect to the uniform convergence topology in $\mathrm{Homeo}(X)$. }    
\end{remark}

\begin{lemma}\label{spatially-defined}
Let $G$ and $H$ be two   countable  infinite residually finite groups, and   let $X$ and $Y$ be free exact odometers of $G$ and $H$, respectively.

If there exists an isomorphism 
\[
\alpha : [[L_X]] \longrightarrow [[L_Y]],
\]
then there exists a homeomorphism 
\[
\phi : X \longrightarrow Y
\]
such that
\[
\alpha(f) = \phi \circ f \circ \phi^{-1}, \quad \forall f \in [[L_{X}]].
\]
In particular, $\alpha$ is continuous with respect to the uniform convergence topology.

The same holds if $L_X$ and $L_Y$ are replaced by $R_X$ and $R_Y$.
\end{lemma}
\begin{proof}
Each orbit of the (left or right) action of $X$ on itself has cardinality $|X|$, and the topological full groups $[[L_{X}]]$ and $[[R_X]]$ have many involutions (see \cite[Definition 1]{M11}).  
Therefore, by \cite[Theorem 2.6 and Remark 3]{M11}, there exists a ho\-meo\-morphism $\phi : X \longrightarrow Y$ with the desired properties.

Let $\varepsilon > 0$, and choose $\delta > 0$ such that 
for all $x, y \in X$ with $d(x,y) \le \delta$, 
we have $d(\phi(x), \phi(y)) \le \varepsilon$. 
If $f, g \in [[L_X]]$ satisfy $d_\infty(f,g) \le \delta$, 
then for every $x \in Y$, $d\bigl(f(\phi^{-1}(x)), g(\phi^{-1}(x))\bigr) \le \delta$. Hence,
\[
d\bigl(\phi \circ f \circ \phi^{-1}(x), \, \phi \circ g \circ \phi^{-1}(x)\bigr) 
\le \varepsilon, \quad \forall x \in Y,
\]
which shows that $d_\infty\bigl(\alpha(f), \alpha(g)\bigr) \le \varepsilon.$
This proves that $\alpha$ is continuous with respect to the uniform convergence topology.   The same holds if we replace $L_X$ by $R_X$.
\end{proof}

\begin{lemma}\label{auxiliar-bis}
Let $G$ and $H$ be two infinite finitely generated residually finite groups. Let $(\Gamma_n)_{n\in\mathbb{N}}$ and $(\Lambda_n)_{n\in\mathbb{N}}$ be scales of $G$ and $H$, respectively,   and let $X$ and $Y$ denote the corresponding free exact odometers.   Let  $\alpha: [[R_{X}]] \to [[R_{Y}]]$ be an isomorphism. Then  there exists $m\in \mathbb{N}$ such that 
$$
\alpha(R_{X}^{\xi})\in [[R_{Y}]]_{m}, \mbox{ for every } \xi\in X.
$$
  The same holds if $R_X$ and $R_Y$ are replaced by $L_X$ and $L_Y$.
\end{lemma}
\begin{proof}
  Lemma \ref{spatially-defined} implies there exists a homeomorphism $\phi:X\to Y$ such that $\alpha(f)=\phi\circ f\circ \phi^{-1}$, for every $f\in [[R_{X}]]$. Since $G$ is finitely generated, there exist $g_1,\cdots, g_k\in G$ such that $G=\langle g_1,\cdots, g_k\rangle$. Let $m\in \mathbb{N}$ be such that $\alpha(R_{X}^{g_i})\in [[R_{Y}]]_{m}$, for every $1\leq i\leq k$. Since $[[R_{Y}]]_{m}$ is a subgroup, we have $\alpha(R_{X}^g)\in [[R_{Y}]]_{m}$, for every $g\in G$.  Let $\xi \in X$ and $(g_i)_{i\in\mathbb{N}}$ be a sequence of elements of $G$ that converges to $\xi$.  This implies that $(R_{X}^{g_i})_{i\in\mathbb{N}}$ converges uniformly to $R_{X}^{\xi}$. From Lemma  \ref{spatially-defined} we get that $(\alpha(R_{X}^{g_i}))_{i\in\mathbb{N}}$ converges uniformly to $\alpha(R_{X}^{\xi})$. Lemma \ref{closure-1} implies that $[[R_{Y}]]_{m}$ is closed with respect to the uniform topology, from which we get $\alpha(R_{X}^{\xi})\in [[R_{Y}]]_{m}$.

The proof for  $L_X$  is similar.
\end{proof}

\begin{lemma}\label{measures}
  Let $G$ be an infinite residually finite group.   Let $X$ be a free exact odometer of $G$.    The Haar measure $\mu$ of $X$ is invariant with respect to the actions of     $[[L_{X}]]$ and $[[R_{X}]]$. 
\end{lemma}
\begin{proof}
The Haar measure $\mu$  of $X$ is invariant by left and right multiplication, therefore, it is also invariant with respect to the actions   of $[[L_{X}]]$ and $[[R_{X}]]$. 
    \end{proof}

\subsection{A necessary condition for topological full groups to be isomorphic.}

The following result demonstrates that if    $X$  and $Y$  are  free exact odometers such that  the topological full groups $[[R_{X}]]$ and $[[R_{Y}]]$ are isomorphic, then $X$ and $Y$ admit   clopen subgroups that are isomorphic and have the same finite index.  

\begin{proposition}\label{necessary}
  Let $G$ and $H$ be infinite finitely generated residually finite groups. Let $X_1$ and $X_2$ be  free exact odometers of $G$ and $H$, respectively.    If $[[R_{X_1}]]$ and $[[R_{X_2}]]$ are isomorphic, then  there exist clopen subgroups $U_1\subseteq X_1$, $U_2\subseteq X_2$, and $m\in\mathbb{N}$, such that $U_1$ and $U_2$ are isomorphic as topological groups, and  $[X_1:U_1]=[X_2:U_2]=m$.   
\end{proposition}

\begin{proof}
  Let $(\Gamma_n)_{n\in\mathbb{N}}$ and $(\Lambda_n)_{n\in\mathbb{N}}$ be scales of $G$ and $H$, respectively, such that $X_1=G_{(\Gamma_n)}$ and $X_2=H_{(\Lambda_n)}$. 
  
   Let $\alpha:[[R_{X_1}]]\to [[R_{X_2}]]$ be an isomorphism. Lemma \ref{spatially-defined} implies there exists a homeomorphism $\phi:X_1\to X_2$ such that $\alpha(f)=\phi\circ f \circ \phi^{-1}$, for every $f\in [[R_{X_1}]]$. From Lemma \ref{auxiliar-bis},   we have that there exists an $n\in\mathbb{N}$ such that $\alpha(R_{X_1}^{\xi})\in [[R_{X_2}]]_{n}$, for every $\xi\in X_1$.

  Let us define
 $$U_1=\{\xi\in X_1: \exists \eta\in [1_H]_n, \alpha(R_{X_1}^{\xi})|_{[1_H]_n}(x)=x\eta^{-1}, \forall x\in [1_H]_n\}.$$

Observe that $1_G\in U_1$. Furthermore, if $\xi_1$ and $\xi_2$ are in $U_1$, then there exist $\eta_1,\eta_2\in [1_H]_n$ such that $\alpha(R_{X_1}^{\xi_1})|_{[1_H]_n}(x)=x\eta_1^{-1}$ and $\alpha(R_{X_1}^{\xi_2})|_{[1_H]_n}(x)=x\eta_2^{-1}$, for every $x\in [1_H]_n$. Thus, for every $x\in [1_H]_n$ we have
\begin{equation}\label{eq-bis}
\alpha(R_{X_1}^{\xi_1\xi_2^{-1}})(x)=\alpha(R_{X_1}^{\xi_1}\circ (R_{X_1}^{\xi_2})^{-1})(x)=\alpha(R_{X_1}^{\xi_1})\circ\alpha(R_{X_1}^{\xi_2})^{-1}(x)=x\eta_2\eta_1^{-1},
\end{equation}
because $x\eta_2\in [1_H]_n$. This shows that $U_1$ is a subgroup of $X_1$.  

On the other hand, let $\xi\in U_1$ and $\eta\in [1_H]_n$ be such that $\alpha(R_{X_1}^{\xi})|_{[1_H]_n}(x)=x\eta^{-1}$, for every $x\in [1_H]_n$. We have 
\begin{equation}\label{eq-bis-2}
\eta^{-1}=\alpha(R_{X_1}^{\xi})(1_H)=\phi\circ R_{X_1}^{\xi}\circ\phi^{-1}(1_H)=\phi(\phi^{-1}(1_H)\xi^{-1}).
\end{equation}
Let define $\psi:X_1\to X_2$ by
$$\psi(\xi)=\left(\phi(\phi^{-1}(1_H)\xi^{-1})  \right)^{-1}, \mbox{ for every } \xi\in X_1.$$
Equation (\ref{eq-bis-2}) implies $\psi(U_1)\subseteq [1_H]_n$. 
Since $\phi$ is one-to-one, the map $\psi$ is injective. Furthermore, given $\eta\in X_2$, we have $\eta=\psi(\xi)$, with $\xi=\left( \phi^{-1}(\eta)\right)^{-1}\phi^{-1}(1_H)$. If $\eta\in [1_H]_n$ then $\psi^{-1}(\eta)\in U_1$, which shows that $\psi|_{U_1}:U_1\to [1_H]_n$ is a well defined bijection.  From equation (\ref{eq-bis}), we have $\psi(\xi_1\xi_2)=\psi(\xi_1)\psi(\xi_2)$, from which we conclude that $\psi|_{U_1}$ is an isomorphism between $U_1$ and $U_2=[1_H]_n$. Since $\psi$ is a homeomorphism, the group $U_1$ is clopen, because $[1_H]_n$ is clopen.

Finally, observe that for every $\eta\in X_2$,
$$
\psi^{-1}([\eta]_n)=\phi^{-1}([\eta^{-1}]_n)^{-1}\phi^{-1}(1_H).
$$
Thus, if $\mu_1$ is the Haar measure of $X_1$, we have
\begin{eqnarray*}
\mu_1(\psi^{-1}([\eta]_n))=\mu_1(\phi^{-1}([\eta^{-1}]_n)^{-1}) & =  & \mu_1(\phi^{-1}([\eta^{-1}]_n))\\ &= & \mu_1(\phi^{-1}([1_H]_n\eta^{-1}))\\ & = &\mu_1(f(\phi^{-1}([1_H]_n))),
\end{eqnarray*}
where $f\in [[R_{X_1}]]$ is such that $\alpha(f)(x)=x\eta^{-1}$, for every $x\in X_2$. Thus,   from  Lemma \ref{measures} we get
$$
\mu_1(\psi^{-1}([\eta]_n))=\mu_1(\phi^{-1}([1_H]_n)).
$$
Since $\{\psi^{-1}([a]_n): a\in H/\Lambda_n\}$ is a clopen partition of $X_1$, we have
$$
1=\sum_{a\in H/\Lambda_n}\mu(\psi^{-1}([a]_n))=[H:\Lambda_n]\mu_1(\phi^{-1}([1_H]_n)).
$$
This implies 
$$
\mu_1(U_1)=\mu_1(\psi^{-1}([1_H]_n))=\frac{1}{[H:\Lambda_n]},
$$
which implies that $[X_1:U_1]=[H:\Lambda_n]=[X_2: U_2]<\infty$.
\end{proof}

\subsection{A sufficient condition for topological full groups to be isomorphic.}


Let $G$ be a   countable infinite  residually finite group. Let $X=G_{(\Gamma_n)}$ be the   free exact odometer associated with the scale $(\Gamma_n)_{n\in\mathbb{N}}$ of $G$.     Let $(U_n)_{n\in\mathbb{N}}$ be a decreasing sequence of clopen subgroups of $X$ such that $\bigcap_{n\in\mathbb{N}}U_n=\{1_X\}$. 

For every $m\in \mathbb{N}$, consider $X/U_m=\{gU_m: g\in X\}$ and let define
$$[[L_{X}]]_{U_m}=\{f\in [[L_{X}]]: \forall a\in X/U_m, \exists \xi_a\in X, \mbox{ such that } f|_{aU_m}(x)=\xi_a x, \forall x\in aU_m\},$$
and
$$
[[L_{X}]]_{U_m}^0=\{ f\in [[L_{X}]]_{U_m}: \forall a\in X/U_m, \exists \xi_a\in aU_ma^{-1}, \mbox{ such that } f|_{aU_m}=L_{X}^{\xi_a}|_{aU_m} \}.
$$
Here, $[[L_{X}]]_{U_m}$ consists of those elements of the topological full group 
that, on each atom of the partition $\mathcal{P}=\{ a U_m : a \in X/U_m \}$, 
coincide with the action of some element in $X$.  
The subgroup $[[L_{X}]]_{U_m}^0$ contains the elements of $[[L_{X}]]_{U_m}$ that leave  each atom of $\mathcal{P}$  invariant.  
In particular, we have $[[L_{X}]]_{[1_X]_m} = [[L_{X}]]_{m}$.

Analogous definitions can be made for $[[R_{X}]]_{U_m}$ and $[[R_{X}]]_{U_m}^0$ 
by using the partition into right cosets.

Let $m \in \mathbb{N}$, and    let $F_m$ be selector from the cosets $gU_m$, ($g\in X$) selecting $1_X$ from $U_m$.  If $f \in [[L_{X}]]_{U_m}^0$, then for each $a \in F_m$ there exists a unique 
$w \in U_m$ such that
\[
f|_{aU_m}(x) = a w a^{-1} x, \quad \forall x \in a U_m.
\]

This allows us to define, for every permutation $\sigma \in \sym(F_m)$, 
a map
\[
\Phi_\sigma : [[L_{X}]]_{U_m}^0 \longrightarrow [[L_{X}]]_{U_m}^0
\]
by
\[
\Phi_\sigma(f)|_{aU_m}(x) = a w_{\sigma^{-1}(a)} a^{-1} x, 
\quad \forall x \in a U_m, \; a \in F_m, \; f \in [[L_{X}]]_{U_m}^0,
\]
where $w_b \in U_m$ is determined by 
\[
f|_{bU_m}(x) = b w_b b^{-1} x, \quad \forall x \in b U_m, \; b \in F_m.
\]

The map $\Phi_\sigma$ defines an automorphism of $[[L_{X}]]_{U_m}^0$, 
and the map
\[
\Phi_m : \sym(F_m) \longrightarrow \aut([[L_{X}]]_{U_m}^0), 
\quad \Phi_m(\sigma) = \Phi_\sigma,
\]
is a group homomorphism.

\begin{lemma}\label{direct-limit}
Let $G$ be a   countable infinite  residually finite group. Let $X=G_{(\Gamma_n)}$ be the  free exact odometer associated to the scale $(\Gamma_n)_{n\in\mathbb{N}}$ of $G$.     Let $(U_n)_{n\in\mathbb{N}}$ be a decreasing sequence of clopen subgroups of $X$ such that $\bigcap_{n\in\mathbb{N}}U_n=\{1_X\}$. Then

For every $m\in \mathbb{N}$, 
\begin{equation}\label{eq-iso}
[[L_{X}]]_{U_m} \cong [[L_{X}]]_{U_m}^0\rtimes_{\Phi_n}\sym(X/U_m).
\end{equation}
 Furthermore, $[[L_{X}]]$ is isomorphic to the direct limit 
\begin{equation}\label{eq-iso-2}
[[L_{X}]]_{U_1}^0\rtimes_{\Phi_1}\sym(X/U_1) \xrightarrow{\phi_1}[[L_{X}]]_{U_2}^0\rtimes_{\Phi_2}\sym(X/U_2) \xrightarrow{\phi_2}   \cdots, 
\end{equation}
where $\phi_m= p_{m+1}\circ i_m \circ p_m^{-1}$, $i_m:[[L_{X}]]_{U_m}\to [[L_{X}]]_{U_{m+1}}$ is the inclusion map, and $p_m$ is the isomorphism given in (\ref{eq-iso}).
\end{lemma}
\begin{proof}  For every $n\in \mathbb{N}$,   let $F_n$ be selector from the cosets $gU_n$, ($g\in X$) selecting $1_X$ from $U_n$.  From \cite[Proposition 2.3.2]{R10} we have $|F_n|<\infty$. 

Let $n\in\mathbb{N}$ and $f\in [[L_X]]_{U_n}$. 
For every $a\in F_n$, let $\xi^f_a\in X$ be  such that $f|_{aU_n}(x)=\xi^f_a x$, for every  $x\in aU_n$. Let $\sigma_f(a) \in F_n$ be such that $\xi^f_a a\in \sigma_f(a)U_n$. We have $\sigma_f\in \sym(F_n)$, due to $f(aU_n)=\sigma_f(a)U_n$. Furthermore,
$$
f|_{aU_n}(x)= \xi^f_a x=(\xi^fa)(a^{-1}x)= (\sigma_f(a)u_a^f)a^{-1}x,
$$
where $u_a^f\in U_n$ is such that $\xi^f_a a=\sigma_f(a)u_a^f$.   Let $\psi_{\sigma_f}:X\to X$ be  the map given by $\psi_{\sigma_f}|_{aU_n}(x)=\sigma_f(a)a^{-1}x$, for every $x\in aU_n$ and $a\in F_n$. Then
$$
f=\varphi_{f}\circ \psi_{\sigma_f},
$$
where $\varphi_f|_{bU_n}(y)=bu_{\sigma_f^{-1}(b)}^fb^{-1}y$, for every $y\in bU_n$ and $b\in F_n$. Observe that $\psi_{\sigma_f}\in [[L_X]]$ and $\varphi_f\in [[L_X]]_{U_n}^0$.  

If $g, f\in [[L_X]]_{U_n}$ and $x\in aU_n$, then
\begin{eqnarray*}
g\circ f(x) &  = & \sigma_{g\circ f}(a)u_{\sigma_f(a)}^g u_a^f a^{-1}x\\
& = &\left(\sigma_{g\circ f}(a)u_{\sigma_f(a)}^g u_a^f\sigma_{g\circ f}(a)^{-1}\right)\sigma_{g\circ f}(a)a^{-1}x\\
&=&  \left(\sigma_{g\circ f}(a)u_{\sigma_f(a)}^g\sigma_{g\circ f}(a)^{-1}\right)\left( \sigma_{g\circ f}(a)u_a^f\sigma_{g\circ f}(a)^{-1}\right)\sigma_{g\circ f}(a)a^{-1}x  \\
&=& \varphi_g\circ \Phi_{\sigma_g}(\varphi_f)\circ\psi_{\sigma_g\circ\sigma_f}(x).
\end{eqnarray*}

It follows that the map
\[
p_n : [[L_X]]_{U_n} \longrightarrow [[L_X]]_{U_n}^0 \rtimes_{\Phi_n} \sym(X/U_n), 
\quad p_n(f) = (\varphi_f, \sigma_f), \quad \forall f \in [[L_X]]_{U_n},
\]
is a group homomorphism.  
It is straightforward to verify that $p_n$ is in fact an isomorphism.  

Since the sequence of partitions $\{ a U_n : a \in F_n \}_{n \in \mathbb{N}}$ is nested and generates the topology of $X$, 
we conclude that $[[L_X]]$ is isomorphic to the direct limit given in (\ref{eq-iso-2}).
\end{proof}

\begin{proposition}\label{final-sufficient}
Let $G$ and $H$ be infinite finitely generated residually finite groups.   Let $X_1$ and $X_2$ be free exact odometers of $G$ and $H$, respectively.  If there exist clopen subgroups $U_1\subseteq X_1$ and $U_2\subseteq X_2$,  such that $U_1$ and $U_2$ are isomorphic as topological groups, and $[X_1:U_1]=[X_2:U_2]<\infty$, then $[[L_{X_1}]]$ and $[[L_{X_2}]]$ are isomorphic.
\end{proposition}
\begin{proof}

For $i=1,2$,   let $D_i\subseteq X_i$ be  a  selector from the cosets  $gU_i$, ($g\in X_i$) selecting $1_{X_i}$ from $U_i$. 
Let $\tau:D_1\to D_2$ be a bijection such that $\tau(1_{X_1})=1_{X_2}$.  Let $\phi:U_1\to U_2$ be a topological isomorphism.  We can extend $\phi:X_1\to X_2$ as follows
$$
\phi(x)=\tau(a)\phi(a^{-1}x), \mbox{ for every } x\in aU_1 \mbox{  and  } a\in D_1.
$$
The map $\phi$ is a well-defined homeomorphism, and its restriction to $U_1$ corresponds to the original isomorphism between $U_1$ and $U_2$.  

Let $f\in [[L_{X_1}]]$ be such that if $x\in aU_1$ and $\xi\in X_1$ is such that $f(x)=\xi x$, then $\xi\in aU_1a^{-1}$, for every $a\in D_1$. In other words, $f$ leaves invariant the left cosets of $X_1/U_1$.  Thus, if $y\in \tau(a)U_2$ and $\xi\in aU_1a^{-1}$ is such that $f(\phi^{-1}(y))=\xi\phi^{-1}(y)$, then
\begin{eqnarray*}
 \phi\circ f \circ \phi^{-1}(y) = \phi(\xi\phi^{-1}(y))& = & \tau(a)\phi(a^{-1}\xi\phi^{-1}(y))\\
    &=& \tau(a)\phi(a^{-1}\xi a a^{-1}\phi^{-1}(y))\\
    &=& \tau(a) \phi(a^{-1}\xi a)\phi(a^{-1}\phi^{-1}(y))\\
    &=& \tau(a)\phi(a^{-1}\xi a)\tau(a)^{-1}\tau(a)\phi(a^{-1}\phi^{-1}(y))\\
    &=& \tau(a)\phi(a^{-1}\xi a)\tau(a)^{-1}y.
\end{eqnarray*}
This implies that $\alpha(f)=\phi\circ f \circ \phi^{-1}\in [[L_{X_2}]]$ leaves invariant the left cosets of $X_2/U_2$.


Let $(V_n)_{n\in\mathbb{N}}$ be a decreasing sequence of clopen subgroups of $X_1$  such that $\bigcap_{n\in\mathbb{N}}V_n=\{1_{X_1}\}$, and suppose that $V_1= U_1$. For every $n\in \mathbb{N}$, we have $[X_1:V_n]<\infty$ (see \cite[Proposition 2.3.2]{R10}). The choice of $(V_n)_{n\in\mathbb{N}}$ implies that $\{gV_n: g\in X_1\}$ is finer than $\{gU_1: g\in X_1\}$, and the same is true with respect to $\{g\phi(V_n): g\in X_2\}$ and $\{gU_2: g\in X_2\}$.   For $i=1,2$, let $  F_n^{i}\subseteq X_i$ be a selector from the cosets  $aV_n$, ($a\in X_i$) selecting $1_{X_i}$ from $V_n$. Observe that if $a\in F_n^1$ and $t\in D_1$ are such that $aV_n\subseteq tU_1$, then $\phi(aV_n)=\tau(t)\phi(t^{-1}a)\phi(V_n)$.

Let $f\in [[L_{X_1}]]_{V_n}^0$. Let $y\in X_2$,  $a\in F^1_n$, and $t\in D_1$ such that $\phi^{-1}(y)\in aV_n \subseteq tU_1$. Let $\xi\in aV_na^{-1}$ be such that $f|_{aV_n}(x)=\xi x$, for every $x\in aV_n$. We have
\begin{eqnarray*}
\alpha(f)= \phi\circ f \circ \phi^{-1}(y) = \phi(\xi\phi^{-1}(y))& = & \tau(t)\phi(t^{-1}\xi\phi^{-1}(y))\\
    &=& \tau(t)\phi(t^{-1}\xi t t^{-1}\phi^{-1}(y))\\
    &=& \tau(t) \phi(t^{-1}\xi t)\phi(t^{-1}\phi^{-1}(y))\\
    &=& \tau(t)\phi(t^{-1}\xi t)\tau(t)^{-1}\tau(t)\phi(t^{-1}\phi^{-1}(y))\\
    &=& \tau(t)\phi(t^{-1}\xi t)\tau(t)^{-1}y.
\end{eqnarray*}
This implies that $\alpha(f)|_{\tau(t)\phi(t^{-1}a)\phi(V_n)}(y)=\tau(t)\phi(t^{-1}\xi t)\tau(t)^{-1}y.$ On the other hand, observe that 
\begin{eqnarray*}
(\tau(t)\phi(t^{-1}\xi t)\tau(t)^{-1})(\tau(t)\phi(t^{-1}a)\phi(V_n)) & = &\tau(t)\phi(t^{-1}\xi t)\phi(t^{-1}a)\phi(V_n)\\
   &=& \tau(t)\phi(t^{-1}\xi a )\phi(V_n)\\
   &=& \tau(t)\phi(t^{-1}aa^{-1}\xi a)\phi(V_n)\\
   &=& \tau(t)\phi(t^{-1}a)\phi(a^{-1}\xi a V_n)\\
    &=&\tau(t)\phi(t^{-1}a)\phi(V_n).
\end{eqnarray*}
 This shows that the restriction
\[
\alpha|_{[[L_{X_1}]]^0_{V_n}} : [[L_{X_1}]]^0_{V_n} \longrightarrow [[L_{X_2}]]^0_{\phi(V_n)}
\]
is a well-defined group homomorphism. It is straightforward to verify that $\alpha$ is an isomorphism.

Moreover, since
\[
\alpha\bigl(\Phi_\sigma(f)\bigr) = \Phi_\sigma\bigl(\alpha(f)\bigr), \quad 
\forall f \in [[L_{X_1}]]^0_{V_n}, \; \sigma \in \sym(X_1/V_n),
\]
the map
\[
r_n : [[L_{X_1}]]^0_{V_n} \rtimes_{\Phi_n} \sym(m_n) 
\longrightarrow [[L_{X_2}]]^0_{\phi(V_n)} \rtimes_{\Phi_n} \sym(m_n), 
\quad r_n(f, \sigma) = (\alpha(f), \sigma),
\]
is an isomorphism that is compatible with the direct limit in Lemma~\ref{direct-limit}, where 
$m_n = [X_1 : V_n] = [X_2 : \phi(V_n)].$
\end{proof}

\begin{proof}[Proof of Theorem \ref{main_full_group}]
The map    $\alpha:[[R_{X_1}]]\to [[L_{X_1}]]$ given by $\alpha(f)(x)=(f(x^{-1}))^{-1}$, for every $f\in [[R_{X_1}]]$ and $x\in X$, is an isomorphism. This implies the equivalence between (1) and (2). The equivalence with (3) follows from Propositions \ref{necessary}\ and \ref{final-sufficient}.
\end{proof}

\begin{proposition}\label{COE-1}
 Let $G$ and $H$ be infinite finitely generated residually finite groups. Let   $X$ and $Y$ be free exact odometers of $G$ and $H$, respectively. If $(X,\alpha_X,G)$ and $(Y, \alpha_Y,H)$ are  continuously orbit equivalent, then   $[[L_X]]\cong[[L_{Y}]]$.
\end{proposition}
\begin{proof}
  If $(X,\alpha_X,G)$ and $(Y, \alpha_Y,H)$ are   continuously orbit equivalent, then there exist an isomorphism $\phi:[[\alpha_X]]\to [[\alpha_Y]]$ and a homeomorphism $\varphi: X\to Y$ such that $\phi(f)=\varphi\circ f\circ \varphi^{-1}$ (see \cite{CM16} and \cite[Theorem 2.6 and Remark 3]{M11}). Arguing as in the proof of Lemma~\ref{spatially-defined}, we deduce that $\phi$ is uniformly 
continuous with respect to the uniform topology on $\mathrm{Homeo}(X)$ and therefore extends 
to an isomorphism  $\phi:\overline{[[\alpha_X]]}\to \overline{[[\alpha_Y]]}$. The conclusion then follows from Lemma  \ref{closure-1} (see Remark \ref{closure-1-remark}). 
\end{proof}

The converse of Proposition \ref{COE-1} is not true. This is a consequence of the next result.


\begin{corollary}\label{non-COE-1}
There exist finitely generated residually finite groups $G$ and $H$, with $G$ amenable and $H$ non-amenable, and   free exact odometers $X$ and $Y$ of $G$ and $H$, respectively,  such that $[[L_X]]\cong[[L_Y]]$.
\end{corollary}
\begin{proof}
By \cite{KS23}, there exist finitely generated groups $G$ and $H$, with $G$ amenable and $H$ non-amenable, whose profinite completions \( \widehat{G} \) and \( \widehat{H} \) are isomorphic. Since the groups $G$ and $H$ are finitely generated, the groups $\widehat{G}$ and $\widehat{H}$ are finitely generated as profinite groups. Then, \cite[Theorem 1.1]{NS07} implies that if $\phi: \widehat{G}\to \widehat{H}$ is an isomorphism, then $\phi$ is continuous.  On the other hand, since $G$ and $H$ are finitely generated, the groups $\widehat{G}$ and $\widehat{H}$ are odometers, in particular, they are free exact odometers. Thus, taking $X=U_1= \widehat{G}$ and $Y=U_2=\widehat{H}$, from  Theorem \ref{main_full_group} we deduce that $[[L_X]]$ and $[[L_Y]]$ are isomorphic. 
 \end{proof}

\begin{remark}\label{Remark_1}
{\rm  The odometers  $(X,\alpha_X,G)$ and $(Y,\alpha_Y,H)$  described in Corollary \ref{non-COE-1} are not continuously orbit equivalent, because their topological full groups   $[[\alpha_X]]$ and $[[\alpha_Y]]$} are not isomorphic: one is amenable while the other is not (see \cite[Corollary 4.7]{CM16}).  
\end{remark}  

 \section{Stabilized automorphism groups of  odometers.}\label{sec:stabilized}

 
The purpose of this section is to describe the stabilized subgroup of a  free exact odometer $X$ in terms of the associated scale, in order to show that it coincides with  $[[R_X]]$.

   Let $X$ be the exact odometer associated with the scale $(\Gamma_n)_{n\in\mathbb{N}}$ of the countable residually finite group $G$. Let $n\in \mathbb{N}$.  Since  $\Gamma_n$ is normal, every cylinder set of level $n$ is invariant under the action of $\Gamma_n$.  We denote $\alpha_{[a]_n}$ the action   given by the restriction of   $\alpha_X$ to   $\Gamma_n\times [a]_n$, for every  $a\in G/\Gamma_n$.

 \begin{lemma}\label{min comp}
  Let  $G$ be an infinite residually finite group. Let $(\Gamma_n)_{n\in\mathbb{N}}$ be a scale of $G$, and $X=G_{(\Gamma_n)}$ the associated odometer.  Then for every $n\in\mathbb{N}$, the dynamical system $(X, \alpha_X|_{\Gamma_n},\Gamma_n)$  has exactly $[G:\Gamma_n]$  minimal components, each component  corresponding to a cylinder set of level $n$ of $X$.  Furthermore, for every $a\in G/\Gamma_n$ we have that $([a]_n, \alpha_{[a]_n}, \Gamma_n)$ is conjugate to $([1_X]_n, \alpha_{[1_X]_n}, \Gamma_n)$.  
 \end{lemma}
   \begin{proof}
      Let $n\in \mathbb{N}$.  Since  every cylinder set of level $n$ is invariant under the action of $\Gamma_n$,   each  minimal component of $(X, \alpha_X|_{\Gamma_n}, \Gamma_n)$ is contained in some cylinder set of level $n$. Conversely, the minimality of $(X, \alpha_X, G)$ implies that for $x$ and $y$ in the same cylinder set of level $n$, there exists a sequence $(g_i)_{i\in\mathbb{N}}$ of $G$ such that $(g_ix)_{i\in\mathbb{N}}$ converges to $y$. Since the set of   return times to a cylinder set of level $n$ is equal to $\Gamma_n$, we get that $g_i\in \Gamma_n$, for every sufficiently large $i$. This implies that every cylinder set of level $n$ is minimal with respect to the action of $\Gamma_n$.  From this we deduce that $(X, \alpha_x|_{\Gamma_n}, \Gamma_n)$ has exactly $[G:\Gamma_n]$ minimal components, each one corresponding to a cylinder set of level $n$. 

      Let $a\in G/\Gamma_n$ and $g_a\in G$ such that $g_a\Gamma_n=a$. The map $f:[a]_n\to [1_X]_n$, given by $f(x)=xg_a^{-1}$ for every $x\in X$, is a conjugacy between $([a]_n, \alpha_{[a]_n},\Gamma_n)$ and $([1_X]_n, \alpha_{[1_X]_n}, \Gamma_n)$. 
   \end{proof}

 \begin{lemma}\label{min comp 2}
    Let  $G$ be an infinite residually finite group. Let $(\Gamma_n)_{n\in\mathbb{N}}$ be a scale of $G$, and $X=G_{(\Gamma_n)}$ the associated odometer. Then for every $n\in\mathbb{N}$, and $a\in G/\Gamma_n$,   we have
  $$
  \aut([a]_n, \alpha_{[a]_n},\Gamma_n)=\{ R_X^\xi|_{[a]_n}: \xi\in [1_X]_n\}.
  $$
 \end{lemma}
\begin{proof}
Let $\xi\in [1_X]_n$. It is straightforward to check that $R_X^\xi|_{[a]_n}$ is in\\ $\aut([a]_n, \alpha_{[a]_n},\Gamma_n)$. 

If $\phi\in \aut([a]_n, \alpha_{[a]_n}\Gamma_n)$, then $f\circ \phi \circ f^{-1} \in \aut([1_X]_n, \alpha_{[1_X]_n}, \Gamma_n)$, where $f:[a]_n\to [1_X]_n$ is the conjugacy given by $f(x)=xg_a^{-1}$ for every $x\in X$, where $g_a\in G$ is some element in $[a]_n$. From Lemmas \ref{min comp} and \ref{characterization_automorphism}  we get there exists $\xi\in [1_X]_n$ such that $f\circ \phi \circ f^{-1}(f(x))=f(x)\xi^{-1}$, for every $x\in [a]_n$. Thus we have $f(\phi(x))=\phi(x)g_a^{-1}=xg_a^{-1}\xi^{-1}$, and then $\phi(x)=xg_a^{-1}\xi^{-1} g_a$, for every $x\in [a]_n$. Taking $\eta=g_a^{-1}\xi g_a$, we have $\eta\in [1_X]_n$ and $\phi=R_X^{\eta}|_{[a]_n}$. 
    \end{proof}

   \begin{lemma}\label{simetria} 
Let $(X, \alpha, G)$ be a topological dynamical system and let $H$ be a subgroup of $G$. Let $\{V_i\}_{i\in I}$ be the collection of minimal components of $(X, \alpha, H)$. Then for every $\gamma\in \aut(X, \alpha|_H, H)$ there exists $\sigma\in \sym(I)$ such that $\gamma(V_i)=V_{\sigma(i)}$, for every $i\in I$.  
   \end{lemma}
\begin{proof}
 
Let $\Gamma$ be a subgroup of $G$. Let $C\subseteq X$ be a minimal component of $(X,\alpha|_H, H)$, and let $\gamma\in \aut(X,\alpha, G)$. It is enough to show that if $C$ is a minimal component of $(X, \alpha|_H, H)$ then $\gamma(C)$ is a minimal component as well. If $x\in C$ and $g\in H$ then $g\gamma(x)=\gamma(gx) \in \gamma(C)$, which implies that $\gamma(C)$ is invariant under the action of $H$. This implies that $\gamma$ is a factor map from $(C, \alpha|_{C\times H}, H)$ to the system $(\gamma(C), \alpha|_{\gamma(C)\times H}, H)$, which implies that $\gamma(C)$ is minimal.
\end{proof}

 \begin{lemma}\label{simetria 1 aut Gj} 
Let  $G$ be an infinite residually finite group. Let $(\Gamma_n)_{n\in\mathbb{N}}$ be a scale of $G$, and $X=G_{(\Gamma_n)}$ the associated odometer. Then
 \begin{enumerate}
 \item For every $n\in\mathbb{N}$ and every $f\in \aut(X, \alpha|_{\Gamma_n}, \Gamma_n)$ there exists a unique $\sigma\in \sym(G/\Gamma_n)$ such that $f([a]_n)=[\sigma(a)]_n$, for every $a\in G/\Gamma_n$. 
 
 \item For every $\sigma\in \sym(G/\Gamma_n)$ there exists $f\in \aut(X, \alpha|_{\Gamma_n}, \Gamma_n)$ such that $f([a]_n)=[\sigma(a)]_n$, for every $a\in G/\Gamma_n$.
 \end{enumerate}
Moreover, for every $n\in\mathbb{N}$ we have
$$
\aut(X, \alpha|_{\Gamma_n},\Gamma_n)=\overline{[[\alpha_X]]}_{n}=[[R_{X}]]_{n},
$$
where $\overline{[[\alpha_X]]}_{n}$ is the closure of $[[\alpha_X]]_{n}$ with respect to the uniform convergence  topology in $\mathrm{Homeo}(X)$.
 \end{lemma}

 \begin{proof}
 Let $f\in \aut(X, \alpha|_{\Gamma_n}, \Gamma_n)$. By Lemmas \ref{min comp} and \ref{simetria}, there exists $\sigma\in \sym(G/\Gamma_n)$ such that $f([a]_n)=[\sigma(a)]_n$, for every $a\in G/\Gamma_n$.   The uniqueness of $\sigma$ follows from the fact that the collection of cylinder sets of level $n$ is a partition of $X$.

  Let $n\in\mathbb{N}$ and $\sigma\in \sym(G/\Gamma_n)$. Let   $\{g_a: a\in G/\Gamma_n\}$ be a selector from the cosets $g\Gamma_n$, ($g\in G$) selecting $1_G$ from $\Gamma_n$.
  We identify every $g_a$ with the element $(g_a\Gamma_k)_{k\in\mathbb{N}}\in X$.  We define $f_{\sigma}:X\to X$ as 
  $$f_{\sigma}(x)= xg_a^{-1}g_{\sigma(a)},  \mbox{ if } x\in [g_a]_n.$$
  We have $f_{\sigma}([g_a]_n)\subseteq [g_{\sigma(a)}]_n$, for every $a\in G/\Gamma_n$. On the other hand, if $x\in [g_{\sigma(a)}]_n$ then $y=xg_{\sigma(a)}^{-1}g_a\in [g_a]_n$ and $f_{\sigma}(y)=x$, from which we deduce $f_{\sigma}([g_a]_n)=[g_{\sigma(a)}]_n$, for every $a\in G/\Gamma_n$. This implies that $f_{\sigma}$ is a homeomorphism on each cylinder set of level $n$,   and hence  $f_{\sigma}$ is a homeomorphism on $X$. Finally, observe that if $g\in \Gamma_n$ and $x\in [g_a]_n$, then $gx\in [g_a]_n$ which implies that $f_{\sigma}(gx)=gxg_a^{-1}g_{\sigma(a)}=gf_{\sigma}(x)$, which shows that $f_{\sigma}\in \aut(X, \alpha_X|_{\Gamma_n},\Gamma_n)$.

  Let $f\in \aut(X, \alpha_X|_{\Gamma_n}, \Gamma_n)$ and $\sigma\in \sym(G/\Gamma_n)$ be such that $f([a]_n)=[\sigma(a)]_n$, for every $a\in G/\Gamma_n$. We have $f=(f\circ f_{\sigma}^{-1})\circ f_{\sigma}$, with $f_{\sigma}$ as defined above. Observe that $f\circ f_{\sigma}^{-1}|_{[a]_n}\in \aut([a]_n, \alpha_{[a]_n}, \Gamma_n)$, for every $a\in G/\Gamma_n$. Thus, by Lemma \ref{min comp 2} there exists $\{\xi_a: a\in G/\Gamma_n\}\subseteq [1_X]_n$ such that $f\circ f_{\sigma}^{-1}(x)=x\xi_a^{-1}$, for every $x\in [a]_n$ and $a\in G/\Gamma_n$. Then we have
  $$
  f(x)=xg_a^{-1}g_{\sigma(a)}\xi_{\sigma(a)}^{-1},  \mbox{ if } x\in [a]_n,
  $$
  which implies, by Lemma \ref{closure-1}, that 
  $$
  \aut(X, \alpha_X|_{\Gamma_n}, \Gamma_n)\subseteq [[R_{X}]]_{n}=\overline{[[\alpha_X]]}_{n}.
  $$
  The other   inclusion is straightforward.
   \end{proof}

\begin{corollary}\label{isomorphism}
Let  $G$ be an infinite residually finite group. Let $(\Gamma_n)_{n\in\mathbb{N}}$ be a scale of $G$, and $X=G_{(\Gamma_n)}$ the associated odometer. Then
          $$[[R_X]]_{n}=\aut(X, \alpha_X|_{\Gamma_n}, \Gamma_n)\cong [1_G]_n^{G/\Gamma_n} \rtimes \sym (G/\Gamma_n), \mbox{ for every } n\in\mathbb{N}.$$
\end{corollary}

\begin{proof}
Let $n\in \mathbb{N}$. It is enough to show there is a split exact sequence 
     
    $$1\longrightarrow [1_G]_n^{G/\Gamma_n}\xlongrightarrow{\psi} \aut(X, \alpha_X|_{\Gamma_n},\Gamma_n)\xlongrightarrow{\pi} \sym(G/\Gamma_n) \longrightarrow 1.$$
We define $\pi: \aut(X, \alpha_X|_{\Gamma_n}, \Gamma_n)\to \sym(G/\Gamma_n)$ as $\pi(\gamma)=\sigma$, where  $\sigma\in \sym(G/\Gamma_n)$ is the permutation associated to $\gamma\in \aut(X, \alpha_X|_{\Gamma_n},\Gamma_n)$ given in Lemma \ref{simetria 1 aut Gj}. By Lemma \ref{simetria 1 aut Gj} the map $\pi$ is surjective. Moreover, $\pi$ is a group homomorphism because if $\gamma_1,\gamma_2\in \aut(X, \alpha_X|_{\Gamma_n},\Gamma_n)$, then   
    $$\gamma_1\circ\gamma_2([a]_n)=\gamma_1([\pi(\gamma_2)(a)]_n)=[\pi(\gamma_1)\circ\pi(\gamma_2)(a)]_n,$$
     which means that $\pi(\gamma_1\circ\gamma_2)=\pi(\gamma_1)\circ\pi(\gamma_2).$

The map $\psi: [1_G]_n^{G/\Gamma_n}\to \aut(X, \alpha_X|_{\Gamma_n},\Gamma_n)$ is given by $\psi(\gamma_i: i\in G/\Gamma_n)=f$, where $f(x)=x\gamma_i^{-1}$ if $x\in [i]_n$, for every $x\in X$ and $i\in G/\Gamma_n$. Here we identify every $\gamma_i$ with the respective element in the subgroup $[1_G]_n$ of $X$. The map $\psi$ is a well defined injective homomorphism.   
Note that $\mathrm{Im}(\psi)$ consists precisely of those elements of 
$\aut(X,\alpha_X|_{\Gamma_n},\Gamma_n)$ that fix all cylinder sets of level $n$, that is, 
those whose associated permutation $\sigma \in \sym(G/\Gamma_n)$ is the identity. 
Consequently, $\mathrm{Im}(\psi)=\Ker(\pi)$.  Define a map $s \colon \sym(G/\Gamma_n) \to \aut(X,\alpha_X|_{\Gamma_n},\Gamma_n)$ by 
$s(\sigma)=f_\sigma$, where $f_\sigma$ is the map constructed in the proof of part~(2) of 
Lemma~\ref{simetria 1 aut Gj} for each $\sigma \in \sym(G/\Gamma_n)$. 
The map $s$ is a group homomorphism, and by part~(1) of Lemma~\ref{simetria 1 aut Gj} we have 
$\pi \circ s = \mathrm{id}$. Hence, the exact sequence splits.
\end{proof}

  Remember that $I(G)$ denotes the collection of finite index subgroups of $G$.

\begin{Lema}\label{Lema res escala}
Let  $G$ be an infinite residually finite group. Let $(\Gamma_n)_{n\in\mathbb{N}}$ be a scale of $G$, and $X=G_{(\Gamma_n)}$ the associated odometer. Then for every $H\in I(G)$, there exists $n\in \N$ such that $\aut(X, \alpha_X|_{H}, H)\sub \aut(X, \alpha_X|_{\Gamma_n}, \Gamma_n)$. Furthermore,  $\aut(X,\alpha_X|_H,  H)=\aut(X, \alpha_X|_{\langle H, \Gamma_n\rangle},  \langle H, \Gamma_n\rangle)$, where $\langle H,\Gamma_n\rangle$ is the subgroup of $G$ generated by $H$ and $\Gamma_n$.
\end{Lema}
\begin{proof}
    Let $H\in I(G)$. Observe that  for every $k\in \N$,  we have $H\sub\langle H,\Gamma_{k+1}\rangle\sub \langle H,\Gamma_{k}\rangle$, and then
    $$[G:\langle H,\Gamma_k\rangle]\leq [G:\langle H,\Gamma_{k+1}\rangle]\leq [G:H].$$
    Thus  $\{[G:\langle H,\Gamma_k\rangle]\}_{k\in\N}$  is an increasing sequence of positive integer numbers bounded above by $[G:H]$.   This implies there exist $n,m\in \N$ such that $m=[G:\langle H,\Gamma_k\rangle]$ for every $k\geq n$. On the other hand, for $k\geq n$ we have $\langle H,\Gamma_k\rangle$ is a subgroup of $\langle H,\Gamma_n\rangle$, and then
    $$[\langle H,\Gamma_n\rangle:\langle H,\Gamma_k\rangle]=\frac{[G:\langle H,\Gamma_k\rangle]}{[G:\langle H,\Gamma_n\rangle]}=\frac{m}{m}=1,$$
    from which we deduce $\langle H,\Gamma_k\rangle=\langle H,\Gamma_n\rangle$.

      Let $\phi\in \aut(X, \alpha_X|_{H}, H)$, $g\in \langle H,\Gamma_n\rangle$ and $\varepsilon>0$ . Let $\delta>0$ be such that  for every $x,y\in X,$ such that $d(x,y)<\delta,$  we have $d(\phi(x),\phi(y))<\varepsilon/2$. Let $j\geq n$ be such that $2^{-j}<\min\{\varepsilon/2,\delta\}$. Since $g\in \langle H,\Gamma_n\rangle=\langle H,\Gamma_j\rangle$ and $\Gamma_j$ is normal, there exist $h\in H,g'\in \Gamma_j$ such that $g=g'h$. Since    $g'\in \Gamma_j$, for every $x\in X$ we have that $g'x$ and $x$ are in the same cylinder set of level $j$.  Thus  
    \begin{equation}\label{eq-1}
    d(x,g'x)\leq \min\{\varepsilon/2,\delta\} \mbox{ for every } x\in X.
    \end{equation}
Then for every $x\in X$  
    $$d(g'h x,h x)=d(g'\cdot(h x),h x)\leq \delta,$$
    from which we get
    $$d(\phi(g'h x), \phi(h x))\leq\varepsilon/2.$$
   Furthermore, by (\ref{eq-1}) we have
    $$d(\phi(h\cdot x),g'\cdot \phi(h\cdot x))\leq \varepsilon/2.$$ In this way, using the last two inequalities, we obtain
    \begin{align*}
        d(\phi(g\cdot x),g\cdot \phi(x))&= d(\phi(g'h\cdot x),g'h\cdot \phi(x))\\
        &\leq d(\phi(g'h\cdot x), \phi(h\cdot x))+d(\phi(h\cdot x),g'h\cdot \phi(x))\\
        &\leq \varepsilon/2+d(\phi(h\cdot x),g'\cdot \phi(h\cdot x))\\
        &\leq \varepsilon/2+\varepsilon/2=\varepsilon.
    \end{align*}
    Since $\varepsilon>0$ is arbitrary, we conclude  that  $\phi(g\cdot x)=g\cdot \phi(x)$ for every $x\in X,g\in\langle H,\Gamma_n\rangle$.  In other words,  $\phi\in \aut(X, \alpha_X|_{\langle H,\Gamma_n\rangle}, \langle H,\Gamma_n\rangle)\sub \aut(X, \alpha_X|_{\Gamma_n},\Gamma_n)$. This shows that  $$\aut(X, \alpha_X|_{H}, H)\sub \aut(X, \alpha_X|_{\langle H,\Gamma_n\rangle}, \langle H,\Gamma_n\rangle)\sub   \aut(X, \alpha_X|_{\Gamma_n}, \Gamma_n).$$

  The other inclusion is trivial since $H \subseteq \langle H,\Gamma_n \rangle$, which completes the proof.
\end{proof}

Theorem \ref{main-characterization} is a consequence of the next result. 
\begin{proposition}\label{teo res escala}\label{like-full-group}
   Let  $G$ be an infinite residually finite group. Let $(\Gamma_n)_{n\in\mathbb{N}}$ be a scale of $G$, and $X=G_{(\Gamma_n)}$ the associated odometer. Then
    $$\aut^\infty(X,\alpha_X, G)=\bigcup_{n\in\N} \aut(X, \alpha_X|_{\Gamma_n}, \Gamma_n)=\bigcup_{n\in\mathbb{N}}\overline{[[\alpha_X]]}_{n} =[[R_X]],$$
    where the closure is taken with respect to the uniform convergence topology in $\mathrm{Homeo}(X)$.  
\end{proposition}
\begin{proof}
Lemma \ref{Lema res escala} implies that $\aut^\infty(X, \alpha_X, G)=\bigcup_{n\in\N} \aut(X, \alpha_X|_{\Gamma_n}, \Gamma_n)$. Then, from  Lemma  \ref{simetria 1 aut Gj} it follows that $\aut^\infty(X, \alpha_X,G)=\bigcup_{n\in\mathbb{N}}\overline{[[\alpha_X]]}_{n}=\bigcup_{n\in\mathbb{N}}[[R_{X}]]_{n}=[[R_X]]$.
\end{proof}




\begin{proof}[Proof of Theorem \ref{main-1}]
   This is a consequence of Theorem \ref{main_full_group} and Proposition \ref{like-full-group}.
\end{proof}



The next result is an extension of \cite[Theorem 3.5]{JB24}  and is direct from Lemma \ref{isomorphism} and Proposition \ref{teo res escala}.

\begin{corollary}
 Let  $G$ be an infinite residually finite group. Let $(\Gamma_n)_{n\in\mathbb{N}}$ be a scale of $G$, and $X=G_{(\Gamma_n)}$ the associated odometer. Then $\aut^\infty(X, \alpha_X, G)$ is isomorphic to the direct limit
$$
X \xrightarrow{\phi_1} [1_G]_1^{G/\Gamma_1}\rtimes \sym(G/\Gamma_1) \xrightarrow{\phi_2} [1_G]_2^{G/\Gamma_2}\rtimes \sym(G/\Gamma_2) \xrightarrow{\phi_3} \cdots 
$$
 $$=\lim_{\longrightarrow}\left([1_G]_n^{G/\Gamma_n}\rtimes \sym (G/\Gamma_n), \phi_n\right),$$
    where  $\phi_n$ is an injective homomorphism, for every $n\in\mathbb{N}$.
\end{corollary}

As a consequence, we obtain the following result, which shows that  the stabilized automorphism group preserves the amenability of the odometer.

\begin{corollary}\label{amenability}
 Let  $G$ be an infinite residually finite group. Let $(\Gamma_n)_{n\in\mathbb{N}}$ be a scale of $G$, and $X=G_{(\Gamma_n)}$ the associated free exact odometer. The following are equivalent:
 \begin{enumerate}
  \item $\aut^\infty(X, \alpha_X, G)$  is amenable. 
  \item $X$ is amenable.
  \item $[[L_X]]$ is amenable.
  \item $[[R_X]]$  is amenable.
  \end{enumerate}
\end{corollary}

\begin{proof}
If $X$ is non-amenable then $\aut^\infty(X, \alpha_X, G)$ is not either, because  $\aut^{\infty}(X, \alpha_X, G)$ contains a subgroup isomorphic to $X$. 

If $X$ is amenable,  then each $[1_G]_n$ is amenable, because it is a subgroup of $X$. Since amenability is closed under semi-direct product and direct limit, we conclude that $\aut^{\infty}(X, \alpha_X, G)$ is amenable.
\end{proof}

 \begin{remark}\label{non-amenability}
{\rm From Corollary \ref{non-COE-1} and Proposition \ref{like-full-group} we deduce that the stabilized automorphism group of an odometer does not convey much information about the acting group itself. It is possible to have $G$ amenable and $H$ non-amenable and $G$ and $H$-odometers with isomorphic stabilized groups. 
}
\end{remark}

\section{Invariance properties of the stabilized automorphism group}\label{sec:invariance}

The results from the previous sections suggest a connection between orbit equivalence properties and the stabilized automorphism group. In this section, we investigate this relationship in more depth.


\begin{corollary}\label{sufficient-finitely-generated}
    Let $G$ and $H$ be infinite finitely generated residually finite groups. Let   $X$ and $Y$ be free exact odometers of $G$ and $H$, respectively. If $(X, \alpha_X,G)$ and $(Y, \alpha_Y, H)$ are continuously orbit equivalent, then $$\aut^{\infty}(X, \alpha_X, G)\cong\aut^{\infty}(Y, \alpha_Y, H).$$ The converse is not true. 
\end{corollary}

\begin{proof}
  Corollary~\ref{COE-1} together with Proposition~\ref{like-full-group} shows that if $(X,\alpha_X,G)$ and $(Y,\alpha_Y,H)$ are continuously orbit equivalent, then $\aut^{\infty}(X,\alpha_X,G)\cong \aut^{\infty}(Y,\alpha_Y,H)$. On the other side, Corollary~\ref{non-COE-1}, Remark~\ref{Remark_1}, and Proposition~\ref{like-full-group} imply that there exist free exact odometers which are not continuously orbit equivalent but nevertheless have isomorphic stabilized automorphism groups.
\end{proof}

  \begin{corollary}\label{necessary_1}
Let $G$ and $H$ be infinite finitely generated residually finite groups. Let $(\Gamma_n)_{n\in\mathbb{N}}$ and $(\Lambda_n)_{n\in\mathbb{N}}$ be scales of $G$ and $H$, respectively.   Let $X=G_{(\Gamma_n)}$ and $Y=H_{(\Lambda_n)}$. If $\aut^{\infty}(X, \alpha_X, G)\cong \aut^{\infty}(Y, \alpha_Y, H)$, then 
$$
    \left\langle\left\{\frac{1}{[H:\Lambda_n]}: n\in\mathbb{N}\right\}  \right\rangle=\left\langle\left\{\frac{1}{[G:\Gamma_n]}: n\in\mathbb{N}\right\}  \right\rangle.
    $$
 \end{corollary}
\begin{proof}
By Theorem \ref{main-1}, if $\aut^{\infty}(X, \alpha_X, G)\cong \aut^{\infty}(Y, \alpha_Y, H)$ then there exist clopen subgroups $U_1\subseteq X$, $U_2\subseteq Y$ such that $[X:U_1]=[Y:U_2]$ and a topological isomorphism $\phi:U_1\to U_2$.   By uniqueness of the Haar measure, it follows    that for every clopen subgroup $C\subseteq U_1$ we have
$$\mu_1(C)=\frac{1}{[X:U_1][U_1:C]}=\frac{1}{[Y:U_2][U_2:\phi(C)]}=\mu_2(\phi(C)),$$ 
where $\mu_1$ and $\mu_2$ are the respective Haar measures. This implies that
\begin{eqnarray*}
 \left\langle\left\{\frac{1}{[H:\Lambda_n]}: n\in\mathbb{N}\right\}  \right\rangle & = &\langle\{\mu_2(\phi(C)): C \mbox{ clopen subgroup of } U_1\}\rangle\\
     & = &\langle\{\mu_1(C): C \mbox{ clopen subgroup of } U_1\}\rangle\\
     &=& \left\langle\left\{\frac{1}{[G:\Gamma_n]}: n\in\mathbb{N}\right\}\right\rangle
\end{eqnarray*}
\end{proof}

  If $(X,\alpha,G)$ is a uniquely ergodic minimal Cantor system, then its reduced dimension group $(D_m(\alpha), D_m(\alpha)^+,[1] )$ is isomorphic to $(K, K^+,1)$, where $K$ is the subgroup of $\mathbb{R}$ generated by $\{\mu(C): C\subseteq X \mbox{ clopen }\}$, and $\mu$ is its unique invariant measure (see \cite{GMPS10} for details). If $X$ is an odometer associated with the scale $(\Gamma_n)_{n\in\mathbb{N}}$ of $G$, and $\alpha=\alpha_X$, then $K$ corresponds to the subgroup  generated by $\left\{\frac{1}{[G:\Gamma_n]}: n\in\mathbb{N}\right\}$. Consequently, the characterization of orbit equivalence for $\mathbb{Z}^d$-actions given in \cite[Theorem 2.5]{GMPS10}, together with Corollaries \ref{sufficient-finitely-generated} and \ref{necessary_1}, yields the following result.

 \begin{corollary}\label{OE}
  Let $d_1, d_2\in \mathbb{N}$. Let   $X_1$ and $X_2$ be free odometers of   $\mathbb{Z}^{d_1}$ and $\mathbb{Z}^{d_2}$, respectively.  Consider the following sentences: 
  \begin{enumerate}
   \item   $(X_1, \alpha_{X_1}, \mathbb{Z}^{d_1})$  and $(X_2, \alpha_{X_2},\mathbb{Z}^{d_2})$ are continuously orbit equivalent.
    
   \item $\aut^{\infty}(X_1, \alpha_{X_1}, \mathbb{Z}^{d_1})$  and  $\aut^{\infty}(X_2, \alpha_{X_2},\mathbb{Z}^{d_2})$ are isomorphic.
    
   \item $(X_1, \alpha_{X_1}, \mathbb{Z}^{d_1})$  and $(X_2, \alpha_{X_2},\mathbb{Z}^{d_2})$ are  orbit equivalent.
  \end{enumerate}
  We have a chain of implications: condition (1) implies condition (2) that implies condition (3).
   \end{corollary}

Combining Corollary~\ref{OE} and the fact that continuous orbit equivalence, orbit equi\-va\-lence, and conjugacy  all coincide within the class of \( \mathbb{Z} \)-odometers, we recover one of the results stated in \cite[Corollary~5.2]{BJ24} and \cite{JB24}.

\begin{corollary}\label{E}
 Two $\mathbb{Z}$-odometers have isomorphic stabilized groups if and only if they are conjugate.   
\end{corollary}

If $d > 1$, Corollary~\ref{E} fails for $\mathbb{Z}^d$-odometers, since there are examples of such odometers that are continuously orbit equivalent but not conjugate (see, e.g., \cite{GPS19}). Furthermore, orbit equivalence does not determine the stabilized automorphism group of a $\mathbb{Z}^d$-odometer, and having isomorphic stabilized automorphism groups does not imply continuous orbit equivalence, as we shall see in the next section.


\subsection{Examples in the context of $\mathbb{Z}^2$-odometers.}
For simplicity, in this section we write $\mathbb{Z}_{(p^n\mathbb{Z})}=\mathbb{Z}_p$,  for every $p\in \mathbb{N}$.

  All odometers appearing in this section are finitely generated profinite groups (see \cite{R10} for background on profinite groups). Hence, by \cite[Theorem~1.1]{NS07}, any group homomorphism between such groups is automatically continuous.

\subsubsection{Isomorphism of stabilized automorphism groups does not imply continuous orbit equivalence.} This example appears in \cite[Example 5.10.3]{GPS19}.  For every $n\in\mathbb{N}$, let
$$
\Gamma_n=\langle\{ \left(\begin{smallmatrix}
     2^n\\0
 \end{smallmatrix}\right), \left(\begin{smallmatrix}
     0\\15^n
 \end{smallmatrix}\right)  \}\rangle   \hspace{10mm} \Lambda_n=\langle\{ \left(\begin{smallmatrix}
     10^n\\0
 \end{smallmatrix}\right), \left(\begin{smallmatrix}
     0\\3^n
 \end{smallmatrix}\right)  \}\rangle.
$$
  Let $X=\mathbb{Z}^2_{(\Gamma_n)}$ and $Y=\mathbb{Z}^2_{(\Lambda_n)}$. The $\mathbb{Z}^2$-odometers $(X, \alpha_X,\mathbb{Z}^2)$ and $(Y, \alpha_Y,\mathbb{Z}^2)$ are orbit equivalent but not continuously orbit equivalent (\cite[Example 5.10]{GPS19}).   However, they are isomorphic as topological groups. In fact, both odometers are isomorphic   as topological groups  to $\mathbb{Z}_{2}\times \mathbb{Z}_{3}\times\mathbb{Z}_{5}$. Thus, from Theorem \ref{main-1} we deduce that $\aut^{\infty}(X, \alpha_X, \mathbb{Z}^2)\cong  \aut^{\infty}(Y, \alpha_Y, \mathbb{Z}^2)$. 

 \subsubsection{Orbit equivalence does not imply isomorphic stabilized automorphism groups.} For every $n\in\mathbb{N}$, let

 $$
\Gamma_n=\langle\{ \left(\begin{smallmatrix}
     6^n\\0
 \end{smallmatrix}\right), \left(\begin{smallmatrix}
     0\\10^n
 \end{smallmatrix}\right)  \}\rangle   \hspace{10mm} \Lambda_n=\langle\{ \left(\begin{smallmatrix}
     6^n\\0
 \end{smallmatrix}\right), \left(\begin{smallmatrix}
     0\\5^n
 \end{smallmatrix}\right)  \}\rangle.
$$
  Let $X=\mathbb{Z}^2_{(\Gamma_n)}$ and $Y=\mathbb{Z}^2_{(\Lambda_n)}$. Since $[\mathbb{Z}^2: \Gamma_n]=60=2^{2n}\cdot 3^n\cdot5^n$ and  $[\mathbb{Z}^2: \Lambda_n]=30^n=2^n\cdot 3^n\cdot 5^n$,  the reduced dimension groups of $(X,\alpha_X,\mathbb{Z}^2)$ and $(Y,\alpha_Y,\mathbb{Z}^2)$ are isomorphic as dimension groups. Thus, from \cite[Theorem 2.5]{GMPS10}  we deduce that $(X, \alpha_X,\mathbb{Z}^2)$ and $(Y, \alpha_Y, \mathbb{Z}^2)$ are orbit equivalent.

 On the other hand, $$X\cong (\mathbb{Z}_2\times \mathbb{Z}_2\times \mathbb{Z}_3\times \mathbb{Z}_5)\cong \mathbb{Z}_2\times \mathbb{Z}_{30}$$ and 
 $$Y\cong (\mathbb{Z}_2\times \mathbb{Z}_3\times \mathbb{Z}_5)\cong \mathbb{Z}_{30}.$$
 
   We may identify $X$ and $Y$ as topological groups with  $\mathbb{Z}_2\times \mathbb{Z}_{30}$ and $\mathbb{Z}_{30}$, respectively. If $\aut^{\infty}(X, \alpha_X, \mathbb{Z}^2)\cong \aut^{\infty}(Y, \alpha_Y, \mathbb{Z}^2)$, then Theorem \ref{main-1} implies there exist isomorphic clopen subgroups $U_1\subseteq X$ and $U_2\subseteq Y$ with the same finite index.  We can assume that $U_1$ is isomorphic to the product of cylinder sets of level $n$, for some $n\in\mathbb{N}$. Thus, $U_1$ is isomorphic to  $X$.  Since $U_2$ is isomorphic to $U_1$, we get that $Y$ contains a clopen subgroup isomorphic to $X$. However, the minimal number of generators of $X$ is $2$, while the minimal number of generators of $Y$ is $1$, which is a contradiction (see \cite[Proposition 4.3.6]{R10}). Thus   $\aut^{\infty}(X, \alpha_X, \mathbb{Z}^2)$ and $\aut^{\infty}(Y, \alpha_Y, \mathbb{Z}^2)$ are not isomorphic.
 
 


\subsubsection{Isomorphism of stabilized automorphism groups does not imply isomorphism as topological groups.}
The $\mathbb{Z}^2$-odometers of \cite[Example 3.7]{CM16} are continuously orbit equivalent, then they have isomorphic stabilized automorphism groups. However, one of the odometers is isomorphic as a topological group to $\mathbb{Z}_{(2^n\mathbb{Z})}\times \mathbb{Z}_{(2^n\mathbb{Z})}$, meanwhile the other one is isomorphic to $\mathbb{Z}_{(4^n\mathbb{Z})}$.

\begin{proposition}\label{Summary}
    In the context of $\mathbb{Z}^2$-odometers: 
\begin{enumerate}
\item Orbit equivalence does not imply isomorphism of stabilized automorphism groups.
\item Isomorphism of stabilized automorphism groups does not imply continuous orbit equivalence.
\item Isomorphism of stabilized automorphism groups does not imply Isomorphism as topological groups.

\end{enumerate}
\end{proposition}

\end{document}